\documentclass[letterpaper, bothsides, 12pt]{article} % Configuración del documento
\usepackage[letterpaper,left=2.5cm,top=2.5cm,right=2.5cm,bottom=2.5cm]{geometry}  % Margenes
\usepackage[utf8]{inputenc}
\usepackage{latexsym} % Símbolos
\usepackage{layout}
\usepackage{float}
\usepackage{textcomp}
\usepackage{amstext}
\usepackage[pdftex]{graphicx}
\usepackage{multicol} % Texto a varias columnas
\usepackage{longtable} % Tablas largas verticalmente
\usepackage{lscape} % Orientación horizontal
\usepackage{multirow} % Tablas multifila
\usepackage{array} % Entorno matricial
\usepackage{amssymb}
\usepackage{amsmath} % Formas matemáticas
\usepackage{amsthm}
\usepackage[T1]{fontenc} % Tildar
\usepackage{listings} % Algoritmos
\usepackage[center, small]{caption2} %rótulos de imagenes y tablas 
\usepackage{graphicx}
\usepackage{verbatim}

\newtheorem{theorem}{Theorem}[section]

\theoremstyle{definition}

\theoremstyle{remark}

\theoremstyle{notation}
\newtheorem{notation}[theorem]{Notation}

\theoremstyle{prop}

\theoremstyle{coexample}

\usepackage{titlesec}
\theoremstyle{definition}
\usepackage[pdftex]{color}
\usepackage[pdftex,colorlinks=true,linkcolor=blue,citecolor=black,urlcolor=cyan]{hyperref}

\newlength{\LL}

\title{ \Huge \bf Hessian estimates for Lagrangian mean curvature equation with sharp Lipschitz phase}
\author{Xingchen Zhou\footnote{zxc3zxc4zxc5@stu.xjtu.edu.cn}
~~ \\ {\it Department of Mathematical Sciences} \\ {\it Tsinghua University} 
}
\date{\today}

\begin{document}
\renewcommand{\tablename}{Tabla}

\maketitle

\hrulefill

\begin{abstract}
We establish a prior interior $C^{1,1}$ estimates for convex solutions and supercritical phase solutions to the Lagrangian mean curvature equation with sharp Lipschitz phase. Counter-examples exist when the phase is H\"{o}lder continuous but not Lipschitz. As an application we obtain interior $C^{2,\alpha}$ regularity for $C^0$ viscosity solutions on the first phase interval $((n-2)\frac\pi2,n\frac\pi2)$.
\end{abstract}

%\begin{figure}[H]
%\begin{center}
%\includegraphics[scale=0.7]{lentes_delgadas.jpg} \\
%\caption{Geometria}\label{paraxial}
%\end{center}
%\end{figure}

%\begin{equation}\label{paraxial}
%\frac{1}{s_o}+\frac{1}{s_i}=(n_{lm}-1)\left(\frac{1}{R_1} - \frac{1}{R_2}\right)
%\end{equation}

\section{Introduction}
This work studies the Hessian estimates for the
Lagrangian mean curvature equation with Lipschitz phase $\phi$ in general dimensions $n\ge 3$:
\begin{equation}
\label{gslag1}
 \arctan {\lambda_1}+\arctan {\lambda_2}+\ldots+\arctan {\lambda_n}= \phi(x),
\end{equation}
 where $u$ will denote a solution, $\lambda_1\ge\lambda_2\ldots\ge \lambda_n$ are the eigenvalues of $D^2u$. It is the potential equation for the Lagrangian graph $(x,Du)$ with mean curvature vector $H=J\nabla_g\phi$, where $J$ is the complex structure (Harvey-Lawson \cite{HL82}, Proposition 2.17). When $\phi$ is a constant, the Lagrangian graph is a volume minimizing surface, (\ref{gslag1}) is then called special Lagrangian equation. Following Yuan \cite{Yua06}, we call the phase $|\phi|= \frac\pi2(n-2)$ critical ("$>$" supercritical and "$<$" subcritical) since the level set $\{\lambda\in \mathbb{R}^n|\arctan\lambda=c\}$ is convex only when $|c|\ge \frac\pi2(n-2)$.\\

% It is well known that special Lagrangian equations enjoy good regularity on critical and supercritical phases, while singular solutions have been found on subcritical phases.
% Our results show that, for convex solutions and supercritical phase solutions, (\ref{gslag1}) admits a prior $C^{1,1}$ estimates depending on $\mathrm{osc}\ u$ and $||\phi||_{C^{0,1}}$.

Hessian estimates for (\ref{gslag1}) were obtained on critical and supercritical phases with $\phi\in C^{1,1}$ in Bhattacharya \cite{Bha20a}, Lu \cite{Lu22}. In \cite{Bha20a} the author also gave singular examples with H\"{o}lder continuous $\phi$. Meanwhile, on critical phase even the gradient estimates with Lipschitz $\phi$ are not known, related results with $\phi\in C^{1,1}$ are in Bhattacharya-Mooney-Shankar \cite{BMS22}. Our main theorems are the follows:

%\begin{theorem} \label{Thm_Upbd_n}
%Let $u$ be a smooth solution to equation (\ref{gslag1}) in $B_{1}(0)$. Suppose that either $u$ is convex, or $\phi\ge \frac\pi2(n-2)+\delta$ for some $\delta>0$. Then we have
%$$
%|D^2u(0)|\le C(\delta,||\phi||_{C^{0,1}(B_1(0))})\int_{B_1(0)}\Delta udv_g.
%$$
%\end{theorem}

\begin{theorem}[Convex solution] \label{Thm_Upbd_n}
Let $u$ be a smooth convex solution to equation (\ref{gslag1}) in $B_{4}(0)$. Suppose that $\phi(x)\in[0,\frac\pi2 n)$ is smooth. Then we have
$$
|D^2u(0)|\le C(n, \mathrm{osc}_{B_{4}(0)} u, ||\phi||_{C^{0,1}( B_{4}(0))}).
$$
\end{theorem}
\begin{theorem}[Supercritical phase solution] \label{Thm_Upbd_critical}
Let $u$ be a smooth solution to equation (\ref{gslag1}) in $B_{4}(0)$. Suppose that $\phi(x)\in[\frac\pi2(n-2)+\theta,\frac\pi2 n)$ for some constant $\theta>0$, $\phi(x)$ is smooth. Then we have
$$
|D^2u(0)|\le C(n,\theta, \mathrm{osc}_{B_{4}(0)} u, ||\phi||_{C^{0,1}( B_{4}(0))}).
$$
\end{theorem}
A natural application of these estimates is the interior regularity for $C^0$ viscosity solutions. We have the following theorem:
\begin{theorem}[Interior Regularity on $I_1$] \label{Thm_Upbd_nGloabl}
Let $u$ be a $C^0$ viscosity solution to equation (\ref{gslag1}) in $ B_{1}(0)$. Suppose that $\phi(x)\in(\frac\pi2(n-2),\frac\pi2n)$,  $\phi(x) \in {C^{0,1}(B_1(0))}$, then $u\in C^{2,\alpha}(B_1(0))$ for each $\alpha\in(0,1)$. Further if $\phi(x)\in[\frac\pi2(n-2)+\theta,\frac\pi2 n)$ for some constant $\theta>0$, then
$$
||u||_{C^{2,\alpha}( B_{\frac12}(0))}\le C(n, \alpha,\theta, \mathrm{osc}_{B_{\frac34}(0)} u, ||\phi||_{C^{0,1}(\overline B_{\frac34}(0))}).
$$
\end{theorem}

Equation (\ref{gslag1}) is concave when $u$ is convex, the classical theorems of Evans-Krylov-Safonov (\cite{Eva82} \cite{Kry83} \cite{Saf84} \cite{Saf89}) give $C^{2,\alpha}$ estimates, see also Caffarelli-Yuan \cite{CY00} which extends the convexity condition. Yuan \cite{Yua01} further removes all convexity constrains for three dimensional special Lagrangian equations, using a VMO argument that can be automatically pushed to H\"{o}lder right hand sides. Unfortunately for convex viscosity solutions we lack an effective way of approximation that prevents us from deriving regularity. On supercritical phases, we can do better due to the existence of comparison principle (Dinew-Do-T\^o \cite{DDT19}, Cirant-Payne \cite{CP21} and Harvey-Lawson \cite{HL21}), and correspondingly the solvability of the Dirichlet problem with $C^{2,\alpha}$ solutions (Bhattacharya \cite{Bha20b}, Lu \cite{Lu22}). The a prior interior $C^{2,\alpha}$ estimates come from Safonov \cite{Saf84} (also \cite{Saf89}), in an exquisite proof of assertion \textbf{A)}. Caffarelli \cite{Caf89} included this result using another approach which applies to $W^{2,p}$ estimates as well, with the help of Caffarelli-Yuan \cite{CY00} and Yuan \cite{Yua06}, Lemma 2.1. Both techniques relay on the solvability of the Dirichlet problem for special Lagrangian equations (Caffarelli-Nirenberg-Spruck \cite{CNS85}, Yuan \cite{Yua16}). The $C^{2,\alpha}$ regularity is then obtained by standard approximating process.\\

We remark that gradient estimates for quasilinear mean curvature equations with Lipschitz right hand side were proved in Ladyzhenskaya-Ural'tseva \cite{LU70}, Simon \cite{Sim76}, Trudinger \cite{Tru73}. As a comparison, Theorem \ref{Thm_Upbd_n}-\ref{Thm_Upbd_nGloabl} do not depend on the derivatives of the mean curvature vector (but $\phi\in C^{1,1}$ will do).\\
%
%Equation (\ref{gslag1}) originates in the special Lagrangian geometry . The mean curvature of the Lagrangian graph $(x,Du)$ is $H=J\nabla_g\phi$, where $J$ is the complex structure and $g$ is the induced metric. It can be viewed as the potential equation for the Lagrangian submanifold with mean curvature vector $H$.

Heinz \cite{Hei59} derived a Hessian bound for Monge-Amp\`ere type equation including equation (\ref{gslag1}) in dimension 2, Pogorelov \cite {Pog64} derived the Hessian estimates with $\phi\ge\frac\pi2$. Gregori \cite{Gre94} treated the two dimensional system problem with additional assumption on the mean curvature vector. Pogorelov \cite{Pog78} constructed his famous counter-examples
to three dimensional Monge-Amp\`ere equation $\det D^2u=1$ which serve as counter-examples
for cubic and higher order symmetric $\sigma_k$ equations, see Urbas \cite{Urb90}. Pogorelov \cite{Pog78} and Chou-Wang \cite{CW01} established Hessian estimates to Monge-Amp\`ere equations and $\sigma_k$ equations for $k\ge 2$ with certain strict convexity constraints. Trudinger \cite{Tru84} and
Urbas \cite{Urb00} \cite{Urb01}, also Bao-Chen \cite{BC03} obtained (pointwise) Hessian estimates in
terms of certain Hessian integrals, for $\sigma_k$ equations and special Lagrangian
equation ((\ref{gslag1}) with $n = 3$, $\phi=\pi$) respectively. Bao-Chen-Guan-Ji \cite{BCGJ03} obtained pointwise Hessian estimates to
strictly convex solutions for quotient equations $\sigma_n/\sigma_k$ in terms of
certain integrals of the Hessian. In dimension 3, pointwise Hessian estimates for $(\ref{gslag1})$ with $\phi=\frac\pi2$ were obtained in Warren-Yuan \cite{WYb09}. In general dimensions, Hessian estimates to convex solutions for equation (\ref{gslag1}) with constant phases were obtained in Chen-Warren-Yuan \cite{CYW09}, Hessian estimates on constant critical and supercritical phases were obtained in Wang-Yuan \cite{WY11}. Singular solutions on constant subcritical phases are in Nadirashvili-Vl$\breve{a}$du$\c{t}$ \cite{NV10} and Wang-Yuan \cite{WY13}, and more recently Mooney-Savin \cite{MS23} for non $C^1$ examples. Qiu \cite{Qiu17} also Guan-Qiu \cite{GQ19} proved Hessian estimates for
quadratic Hessian equations with a $C^{1,1}$ variable right hand side. Hessian estimates for semi-convex solutions to $\sigma_2=1$ were
derived by a compactness argument in McGonagle-Song-Yuan \cite{MSY19}. Shankar-Yuan \cite{SY20a} \cite{SY20b} \cite{SY22} proved regularity for almost convex viscosity solutions, semi-convex solutions and semi-convex entire solutions to $\sigma_2=1$ respectively. Mooney \cite{Moo21} proved regularity for convex viscosity solutions to $\sigma_2=1$ using a new approach. Bhattacharya \cite{Bha20a} \cite{Bha22} and Lu \cite{Lu22} proved Hessian estimates for equation (\ref{gslag1}) with $C^{1,1}$ critical and supercritical phases, Bhattacharya-Shankar \cite{BS20} proved $C^{2,\alpha}$ regularity for $C^{1,\beta}$ convex viscosity solutions with $C^\alpha$ phases.\\

To give an heuristic thought of the proof, let us first review the constant phase situation treated in Warren-Yuan \cite{WYb09}. The idea is to find a subharmonic function of $D^2u$, say $b=\ln\sqrt{1+\lambda_{\max}^2}$, and further prove it satisfies a Jacobi inequality. Subharmonicity gives a Michael-Simon mean integral bound of $b$ on the Lagrangian graph, the Jacobi inequality allows the use of Sobolev inequality to reduce $\oint b$ to $g-$ weighted integral by controlling the gradient. An important observation made in Warren-Yuan \cite{WYb09} (Wang-Yuan \cite{WY11} for general dimensions) is that the volume element admits a divergence structure, then an iteration process leads to the final conclusion.\\

The above approach needs the forth order derivatives of $u$, which means a twice differentiation of the equation. When $\phi$ is a constant the process is harmless, in general the second order derivatives of $\phi$ will appear. To eliminate it, an intuitive thought is to assemble a divergence structure of $\phi$, which requires a new subharmonic function of $D^2u$ in large. We find that $\Delta u$ is a promising one that leaves behind the term $\Delta \phi$ in the Jacobi inequality. Under integral sense $\Delta \phi$ can be reduced to $D\phi$, the extra terms are controlled by a positive quadratic form in the trace Jacobi inequality. Then we can go through the Wang-Warren-Yuan process and get the desired estimates.\\

As one can see, the above arguments date back to the original integral proof of the gradient estimates for minimal surface equations by Bombieri-De Giorgi-Miranda \cite{BDM69}, and for elliptic equations in divergence form by Ladyzhenskaya-Ural'tseva \cite{LU70}, Simon \cite{Sim76}. The above works used test function arguments together with a Sobolev inequality. A simplified proof for quasilinear mean curvature equations was given in Trudinger \cite{Tru72} \cite{Tru73}. Also notice that in \cite{Tru73} a Jacobi inequality was introduced for the gradient.\\

\begin{notation}
\indent The following notations will be used throughout this paper.
  \\ \indent  $u_i=\partial_iu,$ $u_{ij}=\partial_{ij}u,$ and $\lambda_{i,j}=\partial_j\lambda_i, \lambda_{i,jk}=\partial_{jk}\lambda_i$, etc.
%\\ \indent $(\mathcal{M},g)$ denotes the Lagrangian graph $(x,Du)\in \mathbb{R}^n\times \mathbb{R}^n$ with induced metric $g$.
% \\ \indent $\Delta_g$ is the Beltrami-Laplace operator on $\mathcal{M}$.
   \\ \indent $V=\Pi^n_{i=1}\sqrt{1+\lambda_i^2}.$
   \\ \indent $C=C(\diamond)$ are constants that depend on $\diamond$ unless specified previously.
%   \\ \indent Einstein summing convention will be used only for indexes $i,j$.
%   \\ \indent $C$ will denote constants depending only on dimension n unless specified additionally.
\end{notation}
%In this section we prove two interior estimates for convex solutions and supercritical solutions respectively. Take the convex solutions as an example, our proof will be in three major steps. The first step is to prove a week trace Jacobi inequality such that the term involves phase acquires a divergence form. In the second and third step we will prove a mean value inequality in which we get rid of the second order derivatives of the phase, and use the Wang-Yuan technique to get a pointwise Hessian bound from the $g-W^{2,1}$ integral. The interior $C^{2,\alpha}$ estimate follows form the concaveness of the equation.

\section{Interior Hessian estimates with Lipschitz phase}
We first introduce the geometric formulas that will be used in our proof. One can refer to the original work of Harvey-Lawson \cite{HL82}, or more recently Bhattacharya \cite{Bha20a}, Section 2.2 for related calculations.\\

Denote $(M,g)=(x,Du)$ the Lagrangian graph with an induced metric $g$ from $(\mathbb{R}^n\times \mathbb{R}^n,dx^2+dy^2$). Choose a coordinate system such that $D^2u$ is diagonal at a point $p\in \mathcal{M}$, the induced metric is
$ g_{ij}=\delta_{ij}+Du_i\cdot Du_j\stackrel{p}{=}(1+\lambda_i^2)\delta_{ij}.$
We also denote $[g^{ij}]=[g_{ij}]^{-1}$ the inverse of the metric.
The Beltrami-Laplace operator on $\mathcal{M}$ is
$\Delta_g =\frac{1}{\sqrt g}\partial_j(\sqrt gg^{ij}\partial_i)\stackrel{p}{=}g^{ii}\partial_{ii}-g^{ii}\lambda_i\phi_i\partial_i.$ The mean curvature of $\mathcal{M}$ is $H=J\nabla_g\phi$ where $J$ is a $\frac\pi 2$ rotation matrix in $\mathbb{R}^n\times \mathbb{R}^n$.

\subsection {Estimates for convex solutions: proof to Theorem \ref{Thm_Upbd_n}}
\begin{proof}
By approximation we suppose that the eigenvalues of $D^2u$ are pairwise distinct at $p\in \mathcal{M}$. Choose a coordinate system such that $D^2u$ is diagonal at $p$,
\begin{equation}\label{eq_2}
\begin{aligned}
\lambda_{\gamma,ii}&=\sum^n_{s,t=1}\frac{\partial\lambda_\gamma}{\partial u_{st}}
                     \frac{\partial^2u_{st}}{\partial x^2_i}+\sum^n_{s,t,k,l=1}\frac{\partial^2\lambda_\gamma}{\partial u_{st}\partial u_{kl}}\frac{\partial u_{st}}{\partial x_i}
                     \frac{\partial u_{kl}}{\partial x_i}\\
   &=\lambda_{i,\gamma\gamma}+2(\sum_{t\neq \gamma}\frac1{\lambda_\gamma-\lambda_t}u^2_{it\gamma}-\sum_{t\neq i}\frac1{\lambda_i-\lambda_t}u^2_{it\gamma}).
\end{aligned}
\end{equation}
Taking derivative of equation (\ref{gslag1}) with respect to $x_\gamma$ twice, we have at $p$

\begin{equation}\label{eq_3}
\sum^n_{i=1}\lambda_{i,\gamma\gamma}g^{ii}=\sum^n_{i=1}2\lambda_i\lambda_{i,\gamma}^2g^{ii}g^{ii}+\phi_{\gamma\gamma}.
\end{equation}
 Summing up (\ref{eq_2})$\cdot g^{ii} $ and (\ref{eq_3}), and denote $\Delta\cdot=\Delta u$, $\overline \Delta_g=g^{ij}\partial_{ij}\stackrel{p}{=}g^{ii}\partial_{ii}$,
\begin{equation}\label{eq_4}
\overline\Delta_g(\Delta\cdot)
=[\sum^n_{i,\gamma=1}2\lambda_ig^{ii}g^{ii}u^2_{ii\gamma}+\sum^n_{\gamma=1}\sum^n_{i\neq t}(\lambda_i+\lambda_t)g^{ii}g^{tt}u^2_{it\gamma}]_h+\Delta\phi.
\end{equation}
\textbf{Step 1.} We prove a weak trace Jacobi inequality. By (\ref{eq_4}) we have
\begin{align*}
|\nabla_g|A+\Delta\cdot||^2= \sum^n_{\gamma=1}g^{\gamma\gamma}(\sum^n_{i=1}u_{ii\gamma})^2,
\end{align*}
\begin{equation}\label{eq_9}
\begin{aligned}
&\overline\Delta_g|A+\Delta\cdot|-(1+\epsilon) \frac{|\nabla_g\Delta\cdot|^2}{A+\Delta\cdot}= [\ ]_h-
\frac{1+\epsilon}{A+\Delta\cdot}\sum^n_{\gamma=1}g^{\gamma\gamma}(\sum^n_{i=1}u_{ii\gamma})^2+\Delta \phi\\
&\ge \sum^n_{\gamma=1}\frac{g^{\gamma\gamma}}{A+\Delta\cdot}[\sum^n_{i=1}\frac{2\lambda_i(A+\Delta\cdot)}{1+\lambda_i^2}u^2_{ii\gamma}
-(1+\epsilon)(\sum^n_{i=1}u_{ii\gamma})^2]+\Delta \phi,
\end{aligned}
\end{equation}
where $\epsilon$ is a small constant. We consider the following quadratic form with the new representation $u_{nn\gamma}=u_{nn\gamma}-\phi_\gamma{ g_{nn}}$ which will give us a convenient equality $\sum^n_{i=1}u_{ii\gamma}g^{ii}=0$ at $p$ for all $1\le \gamma\le n$,
\begin{align}\label{eq_Q_gamma}
Q_{\gamma}&=\sum^n_{i=1}
\frac{2\lambda_i (A+\Delta\cdot)}{(1+\lambda_i^2)}u_{ii\gamma}^2
-(1+\frac14)(\sum^n_{i=1}u_{ii\gamma})^2.
\end{align}

We claim that $Q_\gamma\ge 0$ if we properly choose the constant $A=A(n)$. We will discuss according to the convexity of $u$. Let $C(n)$, $c(n)$ be some large/small constants depending on $n$ that will be decided later.\\

\textbf{Case 1.} $\lambda_n\ge c(n)$, the uniformly convex situation.

%Let $e_i$ be the $i-$th unit vector in $\mathbb{R}^n$, $\eta_i=\sqrt{a_i-a_1}$, $L=\sqrt{1+\epsilon}\sum^n_{i=1}e_i$. Consider the following $n\times n$ matrix,
%$$
%\Lambda=a_nI_n-\sum^{n-1}_{i=1}(a_n-a_i)e_i^Te_i-L^TL.
%$$
%We need $\Lambda\ge 0$. Since the eigenvectors of $\Lambda$ are $\{\sum^{n-1}_{j=1}\alpha_{ij}(-\eta_j)e_j+\beta_i L\}^n_{i=1}$, we solve the following eigenvalue equation,
%
%\begin{equation}\label{eq_matrix_1}
%\det(\left[
%\begin{array}{llll}
% a_1 & 0\cdots & 0 & \eta_1<L,e_1> \\
% 0\cdots & a_i & 0 & \eta_i<L,e_i> \\
% \vdots  & \vdots &  \vdots &  \vdots \\
% 0\cdots & 0   & a_{n-1} & \eta_{n-1}<L,e_{n-1}> \\
% \eta_1<L,e_1> & \cdots  & \eta_{n-1}<L,e_{n-1}>  & a_n-|L|^2
%\end{array}
%\right]-\lambda I_n)=0.
%\end{equation}

Let $e_i$ be the $i-$th unit vector in $\mathbb{R}^n$, $\eta_i=-\sqrt{a_n-a_i}$, $L=\frac{\sqrt5}{2}\sum^n_{i=1}e_i$. Consider the following $n\times n$ symmetric matrix,
\begin{equation}\label{eq_matrix_1}
\Lambda=a_nI_n-\sum^{n-1}_{i=1}(a_n-a_i)e_i^Te_i-L^TL.
\end{equation}
We suppose that $a_n> a_i>0$, $1\le i\le n-1$. The eigenvectors of $\Lambda$ are $\{\sum^{n-1}_{j=1}\alpha_{ij}(-\eta_j)e_j+\beta_i L\}^n_{i=1}$, we solve the following eigenvalue equation,

$$
\det(\left[
\begin{array}{llll}
 a_1\cdots & \cdots & 0 & \eta_1<L,e_1> \\
% 0\cdots & a_i &\cdots & \eta_i<L,e_i> \\
 \vdots   & \ & \vdots  &  \vdots \\
 0\cdots    & \cdots & a_{n-1} & \eta_{n-1}<L,e_{n-1}> \\
 \eta_1<L,e_1> & \cdots  & \eta_{n-1}<L,e_{n-1}>  & a_n-|L|^2
\end{array}
\right]-\lambda I_n)=0.
$$
%\begin{equation}\label{eq_matrix_1}
%\Lambda=a_1I_n+\sum^{n}_{i=2}(a_{i}-a_{1})e_i^Te_i-L^TL.
%\end{equation}
%The eigenvectors of $\Lambda$ are $\{\sum^{n}_{j=2}\alpha_{ij}\eta_je_j+\beta_i L\}^n_{i=1}$, we solve the following eigenvalue equation,
%$$
%\det(\left[
%\begin{array}{llll}
% a_2\cdots & \cdots & 0 & \eta_2<L,e_2> \\
%% 0\cdots & a_i & 0 & \eta_i<L,e_i> \\
% \vdots  & \ &  \vdots &  \vdots \\
% 0\cdots &  \cdots  & a_{n} & \eta_{n}<L,e_{n}> \\
% -\eta_2<L,e_2> & \cdots  & -\eta_{n}<L,e_{n}>  & a_1-|L|^2
%\end{array}
%\right]-\lambda I_n)=0.
%$$
We require that $\Lambda$ has no negative eigenvalue which equivalents to the following inequality,
$$
a_n-|L|^2-\sum^{n-1}_{i=1}\eta_i^2|<L,e_i>|^2a_i^{-1}=a_n-\sum^{n}_{i=1}\frac{a_n}{a_i}|<L,e_i>|^2\ge 0.
$$
%To prove that, all we need is $a_i\ge a_1$, $1\le i\le n$, and
%$$
%a_1-|L|^2-\sum^{n-1}_{i=1}\eta_i^2|<L,e_i>|^2a_i^{-1}=a_1-\sum^{n}_{i=1}\frac{a_1}{a_i}|<L,e_i>|^2\ge 0.
%$$
Now to prove $Q_\gamma\ge0$, it is sufficient to show the above inequality with $a_i=\frac{2\lambda_i(A+\Delta\cdot)}{1+\lambda_i^2}$, assuming that $a_n$ is the largest one. If $a_n=a_i$ for some $i<n$, we may subtract $a_i$ by an arbitrarily small constant. We have for $A=A(n)$ large,
$$
1-\sum^{n}_{i=1}\frac{1}{a_i}|<L,e_i>|^2\ge 1-\frac{5(C(n)+\Delta\cdot)}{8(A+\Delta\cdot)}\ge 0.
$$
\textbf{Case 2.} $\lambda_i\ge c(n)$ for some $1\le i\le n-1$, the non-degenerate convex situation.
%Let $\eta_i=\sqrt{a_{i}-a_{1}}$, $L=\sqrt{1+\epsilon}\sum^n_{i=1}e_i$. Consider the following $n\times n$ matrix,
%$$
%\Lambda=a_1I_n+\sum^{n}_{i=2}(a_{i}-a_{1})e_i^Te_i-L^TL.
%$$
%The eigenvectors of $\Lambda$ are $\{\sum^{n}_{j=2}\alpha_{ij}\eta_je_j+\beta_i L\}^n_{i=1}$, we solve the following eigenvalue equation,
%$$
%\det(\left[
%\begin{array}{llll}
% a_1 & 0\cdots & 0 & -\eta_1<L,e_1> \\
% 0\cdots & a_i & 0 & -\eta_i<L,e_i> \\
% \vdots  & \vdots &  \vdots &  \vdots \\
% 0\cdots & 0   & a_{n-1} & -\eta_{n-1}<L,e_{n-1}> \\
% -\eta_1<L,e_1> & \cdots  & -\eta_{n-1}<L,e_{n-1}>  & a_1-|L|^2
%\end{array}
%\right]-\lambda I_n)=0.
%$$
%We require that $a_i\ge a_1>0$, $1\le i\le n$, and
%$$
%a_1-|L|^2-\sum^{n-1}_{i=1}\eta_i^2|<L,e_i>|^2a_i^{-1}=a_1-\sum^{n}_{i=1}\frac{a_1}{a_i}|<L,e_i>|^2\ge 0.
%$$
Now to prove $Q_\gamma\ge0$, we need to choose another sequence of $\{a_i\}^n_{i=1}$.
We split $\{\lambda_i\}^n_{i=1}$ into two parts,
$$
\left\{
\begin{array}{ll}
 (\lambda_1,\cdots,\lambda_k)    & \lambda_k\ge c(n)   \\
(\lambda_{k+1},\cdots ,\lambda_{n})   &  \lambda_{k+1}< c(n)     \\
\end{array} \right.
.
$$
Recall we have $\sum^n_{i=1} g^{ii}u_{ii\gamma}=0$ by reducing the derivative of the phase $\phi_\gamma$ to $u_{nn\gamma}$, we have the following relation,
\begin{equation}\label{eq_case_22}
\begin{aligned}
\sum^k_{i=1}\frac {u^2_{ii\gamma}}{(1+\lambda_i^2)^2}&\ge \frac1n(\sum^k_{i=1}\frac{ u_{ii\gamma}}{1+\lambda_i^2})^2=\frac1n(\sum^n_{i=k+1}\frac{u_{ii\gamma}}{1+\lambda_i^2})^2\\
& \ge \frac1{n^2}(\sum^n_{i=k+1}u_{ii\gamma})^2-n^2\sum^n_{i=k+1}\frac {\lambda_i^4u^2_{ii\gamma}}{(1+\lambda_i^2)^2}.
\end{aligned}
\end{equation}
We choose $A=A(n)$ large such that $c(n)A\lambda_i\ge n^4$ for $1\le i\le k$, $A\ge n^6c(n)$, by (\ref{eq_case_22}) we get
$$
\sum^k_{i=1}\frac{2c(n)\lambda_i(A+\Delta\cdot)}{1+\lambda_i^2}u^2_{ii\gamma}
+\sum^n_{i=k+1}\frac{2\lambda_i(A+\Delta\cdot)}{1+\lambda_i^2}u^2_{ii\gamma}\ge {2n^2}(\sum^n_{i=k+1}u_{ii\gamma})^2.
$$
Let $a_i=\min[\frac{2(1-c(n))\lambda_i(A+\Delta\cdot)}{1+\lambda_i^2}, 2n^2-1]$ for $1\le i\le k$, $a_{k+1}=2n^2$, we get
$$
Q_\gamma \ge \sum^k_{i=1}a_i u_{ii\gamma}^2 +a_{k+1}(\sum^n_{i=k+1}u_{ii\gamma})^2 -\frac54(\sum^n_{i=1}u_{ii\gamma})^2.
$$
Now we apply the argument for the quadratic form (\ref{eq_matrix_1}) in \textbf{Case 1},
$$
1-\sum^{k+1}_{i=1}\frac{1}{a_i}|<L,e_i>|^2\ge 1-\frac{5(C(n)+\Delta\cdot)}{8(1-c(n))(A+\Delta\cdot)}-\frac1n\ge 0.
$$
\textbf{Case 3.} $\lambda_1< c(n)$, the degenerate convex situation. In this case we immediately get $Q_\gamma\ge 0$ since
$$
(\sum^n_{i=1}u_{ii\gamma})^2=(\sum^n_{i=1}\frac{\lambda_i^2u_{ii\gamma}}{1+\lambda_i^2})^2\le \sum^n_{i=1}\frac{n\lambda_i^4}{(1+\lambda_i^2)^2}u_{ii\gamma}^2
\le \sum^n_{i=1}\frac{A\lambda_i}{(1+\lambda_i^2)}u_{ii\gamma}^2.
$$
The remaining step is to deal with the term $\phi_\gamma$ that we include in $Q_\gamma$. Let $\delta=\delta(n)$ small, for the original $u_{nn\gamma}$,
\begin{equation}\label{eq_15}
\begin{aligned}
&(\sum^{n}_{i=1}u_{ii\gamma})^2
\le (1+\delta ) (\sum^n_{i=1}u_{ii\gamma}-\phi_\gamma g_{nn})^2+\frac {C(n)}\delta|D\phi|^2 g^2_{nn},\\
&u^2_{nn\gamma}
\ge (1-\delta ) (u_{nn\gamma}-\phi_\gamma g_{nn})^2-\frac {C(n)}\delta|D\phi|^2 g^2_{nn}.
\end{aligned}
\end{equation}
Thus (\ref{eq_9}) becomes (with $\epsilon=\epsilon(\frac14)$, $\delta=\delta(\frac14)$ small)
$$
\overline\Delta_g|A+\Delta\cdot|-(1+\epsilon) \frac{|\nabla_g\Delta\cdot|^2}{A+\Delta\cdot}\ge \delta[\ ]_h - {C(n)} |A+\Delta\cdot||D\phi|^2 +\Delta \phi.
$$
There is still a second part of $\Delta_g|A+\Delta\cdot|$, namely
\begin{equation}\label{eq_150}
|(\Delta_g-\overline\Delta_g) (A+\Delta\cdot)|=|\sum^n_{i=1}g^{ii}\lambda_i\phi_iD_i\Delta\cdot|\le \frac\epsilon2 \frac{|\nabla_g \Delta\cdot|^2}{A+\Delta\cdot}+\frac{C(n)}\epsilon|A+\Delta\cdot||D\phi|^2
\end{equation}
with $\epsilon=\epsilon(\frac14)$. Thus we have (with a new $\epsilon$)
\begin{equation}\label{eq_10}
\Delta_g\ln|A+\Delta\cdot|-\epsilon{|\nabla_g\ln|A+\Delta\cdot||^2}\ge \frac{\delta[\ ]_h}{A+\Delta\cdot}+\frac{\Delta \phi}{{A+\Delta\cdot}}-C(n) |D\phi|^2.
\end{equation}
\textbf{Step 2.} We prove a mean value inequality.
We need to eliminate the term $\Delta\phi$. Let ${\mathcal{B}}_1(0)$ be the geodesic ball on $\mathcal{M}$ with radius 1. Let $\varphi(x)\in C_0^\infty(\mathcal{B}_1(0))$ be a nonnegative function. Notice that $\ln|A+\Delta\cdot|>1$ if $A\ge 3$, we have
\begin{align*}
&\int_{\mathcal{B}_1(0)}\ln|A+\Delta\cdot| \Delta_g \varphi dv_g\ge\int_{\mathcal{B}_1(0)}\frac{\delta[\ ]_h\varphi}{A+\Delta\cdot}dv_g -\int_{\mathcal{B}_1(0)}C(n)|D\phi|^2\varphi dv_g\\
&-\int_{\mathcal{B}_1(0)}\frac{D\varphi D\phi}{A+\Delta\cdot} dv_g-\int_{\mathcal{B}_1(0)}\frac{\varphi D\phi D V}{A+\Delta\cdot}dx +\int_{\mathcal{B}_1(0)}\frac{\varphi D\phi D\Delta\cdot}{(A+\Delta\cdot)^2}dv_g\\
&+\int_{\mathcal{B}_1(0)}\epsilon{|\nabla_g\ln|A+\Delta\cdot||^2}\varphi dv_g,
\end{align*}
where $V=\Pi^n_{i=1}\sqrt{1+\lambda_i^2}$. Pointwisely we can choose a coordinate system such that $D^2u$ is diagonal, and

\begin{align*}
|D\phi D V|&=|\sum^n_{i,\gamma=1}\frac{\lambda_i\lambda_{i,\gamma}}{1+\lambda_i^2}V D_\gamma\phi|
\le \frac14\delta\sum^n_{i,\gamma=1} {\lambda_i}g^{ii}g^{ii}u^2_{ii\gamma}V+C(\delta,||\phi||_{C^{0,1} (\mathcal{B}_1(0))})|A+\Delta\cdot| V,
\end{align*}
\begin{align*}
|D\phi D\Delta\cdot|&\le \frac12\epsilon|\nabla_g|A+\Delta\cdot||^2+C(\epsilon,||\phi||_{C^{0,1}(\mathcal{B}_1(0))})|A+\Delta\cdot|^2.
\end{align*}
Then we get for $C_*=C(\epsilon,\delta,||\phi||_{C^{0,1}(\mathcal{B}_1(0))})=C(n,||\phi||_{C^{0,1}(\mathcal{B}_1(0))})$
\begin{equation}\label{eq_11}
\int_{{\mathcal{B}}_1(0)}\ln|A+\Delta\cdot|\Delta_g\varphi dv_g\ge
-\int_{\mathcal{B}_1(0)}\frac{D\varphi D\phi}{A+\Delta\cdot} dv_g-C_*\int_{\mathcal{B}_1(0)} \varphi dv_g.
\end{equation}
Suppose that $Du(0)=0$, let $r=|(x,Du)|$, let $\chi(t)\in C^\infty(\mathbb{R})$ be a non-decreasing function such that $\chi(t)=0$ when $t\le0$, $\chi(t)=1$ when $t>\mu$ for some small constant $\mu>0$. Let $\rho<1$, we choose
$$
\varphi(x)=\int^{+\infty}_rt\chi(\rho-t)dt.
$$
We have $0\le\varphi\le \rho^2\chi(\rho-r)$, $|D\varphi(x)|\le C(n)\rho|A+\Delta\cdot|\chi(\rho-r)$, and
$$
\Delta_g\varphi(x)\le (-n+C(n)\rho|D\phi|)\chi(\rho-r)+\rho\chi'(\rho-r).
$$
Since $\ln|A+\Delta\cdot|>1$, (\ref{eq_11}) becomes
\begin{align*}
&\int_{\mathcal{B}_1(0)}\ln|A+\Delta\cdot|[n\chi(\rho-t)-\rho\chi'(\rho-r)] dv_g\le
C_*\int_{\mathcal{B}_1(0)}\rho \ln|A+\Delta\cdot|\chi(\rho-r)dv_g.
\end{align*}
Multiply the above inequality by $\rho^{-n-1}$,
\begin{align*}
-\frac d{d\rho}\ln[\frac 1{\rho^{n}}\int_{\mathcal{B}_1(0)}\ln|A+\Delta\cdot|\chi(\rho-r)dv_g]\le C_*.
\end{align*}
Integrate it over $(\sigma,\frac12)$ for some $\sigma>\mu$, since $\sigma,\mu$ are arbitrarily small constants,
\begin{equation}\label{eq_12}
\ln|A+\Delta\cdot|(0)\le C_*\int_{\mathcal{B}_1(0)}\ln|A+\Delta\cdot| dv_g.
\end{equation}
\textbf{Step 3.} We use the weak trace Jacobi inequality to get the final result. We substitute the geodesic ball $\mathcal{B}_1(0)$ by ${B}_1(0)\in \mathbb{R}^n$, and choose the smooth cut-off function $\varphi$ such that $\varphi\ge 0$ in ${B}_2(0)$, $\varphi=1$ in ${B}_1(0)$, $|D\varphi|<2$. Repeat the proof at the beginning of {Step 2.} We get for $C_*=C(n,||\phi||_{C^{0,1}({B}_2(0))})$

$$
\int_{{{B}}_2(0)}\Delta_g\ln|A+\Delta\cdot|\varphi^2 dv_g-\frac\epsilon2\int_{{{B}}_2(0)}{|\nabla_g\ln|A+\Delta\cdot||^2}\varphi^2 dv_g\ge
-C_*\int_{{B}_2(0)} \varphi^2dv_g.
$$
Thus we get the gradient estimate for $\Delta\cdot$,
\begin{equation}\label{eq_14}
\int_{{{B}}_2(0)}{|\nabla_g\ln|A+\Delta\cdot||^2}\varphi^2dv_g\le
C_*\int_{{B}_2(0)}  \varphi^2dv_g.
\end{equation}
To get the final result, we will briefly go through the proof in Wang-Yuan \cite{WY11}. By inequality (\ref{eq_12}) and Sobolev inequality (Michael-Simon \cite{MS73}, Theorem 2.1), and the formulation for the mean curvature $H$, we have for $C_0=C_*(\int_{{B}_1(0)}dv_g)^{\frac{2}{n}}$,
\begin{align*}
&\ln|A+\Delta\cdot|(0)\le C_0(\int_{{B}_2(0)}(\varphi^{n+1}\ln^{\frac12}|A+\Delta\cdot|)^{\frac{2n}{n-2}} dv_g)^{\frac{n-2}{n}}\\
&\le C_0 (\int_{{B}_2(0)}\ln|A+\Delta\cdot|\varphi^{2n}\sum^n_{i=1}g^{ii} dv_g+\int_{{B}_2(0)}|\nabla_g\ln|A+\Delta\cdot||^2\varphi^{2n} dv_g).
\end{align*}
The second term is controlled by the volume, see (\ref{eq_14}). For the first term, we use the following equality in Wang-Yuan \cite{WY11}, proof of Theorem 1.1, step 3,
$$
\sum^n_{i=1}\frac V{1+\lambda_i^2}
=c_0+c_1\sigma_1+\cdots+c_{n-1}\sigma_{n-1},
$$
where $|c_i|\le 2n$. The first term thus reduces to
$$
\int_{{B}_2(0)}\ln|A+\Delta\cdot|\varphi^{2n}\sum^n_{i=1}g^{ii} dv_g\le C(n) \int_{{B}_2(0)}\ln|A+\Delta\cdot|(1+
\sigma_1+\cdots\sigma_{n-1})\varphi^{2n} dx.
$$
Here we need the divergence structure of $\sigma_k$,
$$
k\sigma_k(D^2u)=\sum_{i,j}\frac{\partial}{\partial x_i}(\frac{\partial\sigma_k}{\partial u_{ij}}\frac{\partial u}{\partial x_j})=div({L_{\sigma_k}Du}).
$$
We denote $b=\ln|A+\Delta\cdot|$,
$$
\int_{{B}_2(0)}b\sigma_k\varphi^{2n} dx=-\frac1k\int_{{B}_2(0)}<\varphi^{2n} D b+2nb\varphi^{2n-1}D\varphi,L_{\sigma_k}D u>dx.
$$
Since $|<D b,L_{\sigma_k}Du>|\le C(n)|D u| (|\nabla_g b|^2+1)V$, by (\ref{eq_14}) we have,
$$
\int_{{B}_2(0)}b\sigma_k\varphi^{2n} dx\le C(n)||Du||_{L^\infty(B_2(0))}(\int_{{B}_2(0)}b\sigma_{k-1}\varphi^{2n-1}dx + C_*\int_{{B}_2(0)}dv_g).
$$
Notice that $b\le A+\Delta\cdot$, by iteration we get
$$
\ln|A+\Delta\cdot|(0)\le C_*(1+||Du||^n_{L^\infty(B_2(0))})(\int_{{B}_{2}(0)}dv_g)^{1+\frac2n}.
$$
Repeat the above process with $b=1$ to estimate the volume, we get
$$
\ln|A+\Delta\cdot|(0)\le C(n,||Du||_{L^\infty(B_3(0))},||\phi||_{C^{0,1} ({B}_{3}(0))}).
$$
\end{proof}
We remark that trace Jacobi inequality was newly proved for $\sigma_2$ equation in dimension 3 in Qiu \cite{Qiu17}, and in general dimensions to semi-convex solutions in Shankar-Yuan \cite{SY22}. Before, the Jacobi inequality with $\lambda_{\max}$ for special Lagrangian equation was proved in dimension 3 in Warren-Yuan \cite{WYb09}, in general dimensions in Wang-Yuan \cite{WY11}, and for Lagrangian mean curvature equation in Bhattacharya \cite{Bha20a}. The Jacobi inequality with the volume element $V$ for special Lagrangian equation was proved in Chen-Warren-Yuan \cite{CYW09}.

\subsection {Hessian estimates on supercritical phase: proof to Theorem \ref{Thm_Upbd_critical}}

\begin{proof}
We will always assume $\lambda_n<0$. Recall the following properties of equation (\ref{gslag1}) on critical and supercritical phase $\phi\ge\frac\pi2(n-2)$ (Wang-Yuan \cite{WY11}, Theorem 2.1). a) $\lambda_{n-1}\ge |\lambda_n|$; b) $\sigma_k\ge 0$, $1\le k\le n-1$.
We begin with equality (\ref{eq_4}):
\begin{equation}\label{eq_20}
\overline \Delta_g\Delta\cdot
\ge\sum^n_{\gamma=1}[\sum^n_{i=1}2\lambda_ig^{ii}g^{ii}u^2_{ii\gamma}+\sum^n_{i\neq \gamma}2(\lambda_i+\lambda_\gamma)g^{ii}g^{\gamma\gamma}u^2_{ii\gamma}]_\gamma+\Delta\phi.
\end{equation}
Let $c(\theta),C(\theta)$ be small/large constants depending on $\theta$. Notice that the supercritical phase condition implies $\lambda_{n-1}\ge \tan\theta$, $|\lambda_n|<\cot\theta$. Choose a coordinate system such that $D^2u$ is diagonal at a point $p$, redefine $u_{nn\gamma}=u_{nn\gamma}-\phi_\gamma{ g_{nn}}$ as in the convex case, and denote $h_{ii\gamma}=u_{ii\gamma}g^{ii}$, then $\sum^n_{i=1}h_{ii\gamma}=0$ at $p$. We choose a constant $\mu=10^{-2}$. If $|\lambda_n|<\frac{\mu\tan\theta}{4n}$ or $\lambda_{n-1}>\frac{4n\cot\theta}{\mu}$ , we have $\lambda_i-2n|\lambda_n|\ge (1-\frac{\mu}{2})\lambda_i$ when $i\le n-1$. Notice that if $\cot^2\theta<\frac{\mu}{4n}$ the relation holds automatically. Then
$$
\sum^n_{i=1}\lambda_ih^2_{ii\gamma}
\ge \sum^{n-1}_{i=1}(\lambda_i-2n|\lambda_n|)h^2_{ii\gamma}+|\lambda_n|h^2_{nn\gamma}\ge (1-\frac{\mu}2)\sum^{n}_{i=1}|\lambda_i|h^2_{ii\gamma},
$$
\begin{equation}\label{eq_21}
[\ ]_\gamma
\ge\sum^{n}_{i=1}2(1-\mu)
|\lambda_i|g^{ii}g^{\gamma\gamma}u^2_{ii\gamma}+ {\mu}\sum^{n}_{i=1}|\lambda_i|g^{ii}g^{ii}u^2_{ii\gamma}.
\end{equation}
Otherwise $\lambda_{n-1}\le \frac{4n\cot\theta}{\mu}$, $|\lambda_n|\ge \frac{\mu\tan\theta}{4n}$. Let $\hat\lambda_n=\tan(\arctan\lambda_n-\theta)$, $\kappa> 1$ to be chosen later. By property b) we have $\sum^{n-1}_{i=1}\lambda_i^{-1}+\hat\lambda_n^{-1}\le0$, we calculate that
\begin{align*}
\sum^n_{i=1}\lambda_ih^2_{ii\gamma}&=\sum^{n-1}_{i=1}\lambda_ih^2_{ii\gamma}+\kappa\lambda_nh^2_{nn\gamma}+(1-\kappa)\lambda_nh^2_{nn\gamma}\\
&\ge \kappa\lambda_n \sum^{n-1}_{i=1}\lambda_ih^2_{ii\gamma}(\sum^{n-1}_{i=1}\frac1{\lambda_i}+\frac1{\hat\lambda_n}+\frac1{\kappa\lambda_n}
-\frac1{\hat\lambda_n})+(1-\kappa)\lambda_nh^2_{nn\gamma}\\
 &\ge(1-\frac{\kappa\lambda_n}{\hat \lambda_n})\sum^{n-1}_{i=1}\lambda_ih^2_{ii\gamma}+(1-\kappa)\lambda_nh^2_{nn\gamma} \\
 &\ge \frac{\tan^2\theta-(\kappa-1)}{\tan^2\theta+1}\sum^{n-1}_{i=1}\lambda_ih^2_{ii\gamma}+(\kappa-1)(-\lambda_n)h^2_{nn\gamma}
 \ge \hat\epsilon(\theta)\sum^{n}_{i=1}|\lambda_i|h^2_{ii\gamma}.
\end{align*}
Let $C=C(\theta)>\frac{4n\cot\theta}{\mu}$. Since we assume that $\lambda_{n-1}\le \frac{4n\cot\theta}{\mu}$, let $m\le n-2$ be an integer such that $\lambda_m> C(\theta)$, $\lambda_{m+1}\le C(\theta)$. When $\gamma=n$, we have $\lambda_i+\lambda_\gamma\ge (1-\frac\mu2)\lambda_i$ when $i\le m$, and $\lambda_i+\lambda_\gamma\ge \tan\theta$ when $m<i<n$. From (\ref{eq_20}) we get
\begin{equation}\label{eq_22}
[\ ]_n\ge
\sum^{m}_{i=1}2(1-\mu)\lambda_ig^{ii}g^{nn}u^2_{iin}+\hat\epsilon(\theta)\sum^{n}_{i=m+1} g^{ii}g^{nn}u^2_{iin}+\hat\epsilon(\theta)\sum^{n}_{i=1}|\lambda_i|g^{ii}g^{ii}u^2_{iin}.
\end{equation}
When $\gamma<n$, let $0<\kappa\le 1$, we have
\begin{align*}
&\sum^n_{i=1}\lambda_ih^2_{ii\gamma}=\sum^{n-1}_{i\neq \gamma}\lambda_ih^2_{ii\gamma}+\kappa\lambda_\gamma h^2_{\gamma\gamma\gamma}+\lambda_nh^2_{nn\gamma}+(1-\kappa)\lambda_\gamma h^2_{\gamma\gamma\gamma}\\
&\ge (\sum^{n-1}_{i\neq\gamma}\lambda_ih^2_{ii\gamma}+\kappa\lambda_\gamma h^2_{\gamma\gamma\gamma})[1+\sum^{n-1}_{i=1}\frac{\lambda_n}{\lambda_i}+(\kappa^{-1}-1)\frac{\lambda_n}{\lambda_\gamma}]+(1-\kappa)\lambda_\gamma h^2_{\gamma\gamma\gamma}\\
 &\ge(\sum^{n-1}_{i\neq\gamma}\lambda_ih^2_{ii\gamma}+\kappa\lambda_\gamma h^2_{\gamma\gamma\gamma})[\frac{\tan\theta}{\lambda_{n-1}}-(\kappa^{-1}-1)\frac{\cot\theta}{\lambda_\gamma}]_{=\tau(\theta)}+ (1-\kappa)\lambda_\gamma h^2_{\gamma\gamma\gamma}.
\end{align*}
If $\gamma\le m$, we choose $\kappa=\frac12\mu$ then $C(\theta)$ large such that $\tau(\theta)>\frac12\frac{\tan\theta}{\lambda_{n-1}}\ge\frac{\mu}{8n\cot^2\theta}$.\\
If $\gamma>m$, we choose $\kappa=1$. In either case, we have for some $\hat\epsilon=\hat\epsilon(\theta)>0$
\begin{equation}\label{eq_23}
[\ ]_\gamma\ge
\sum^{m}_{i=1}2(1-\mu)\lambda_ig^{ii}g^{\gamma\gamma}u^2_{iin}+\hat\epsilon(\theta)\sum^{n}_{i=m+1}
g^{ii}g^{\gamma\gamma}u^2_{ii\gamma}+\hat\epsilon(\theta)\sum^{n}_{i=1}|\lambda_i|g^{ii}g^{ii}u^2_{ii\gamma}.
\end{equation}
Now similar to (\ref{eq_9}) (\ref{eq_Q_gamma}), we consider the following inequality with $\epsilon=\epsilon(n)$ small
$$
\overline\Delta_g|A+\Delta\cdot|-(1+\epsilon) \frac{|\nabla_g\Delta\cdot|^2}{A+\Delta\cdot}
\ge \sum^n_{\gamma=1}\{[\ ]_\gamma
-(1+\epsilon)\frac{g^{\gamma\gamma}}{A+\Delta\cdot}(\sum^n_{i=1}u_{ii\gamma})^2\}+\Delta \phi,
$$
and the quadratic form
\begin{equation}\label{eq_24}
Q_{\gamma}=\sum^n_{i=1}
a_iu_{ii\gamma}^2
-(1+\frac18)(\sum^n_{i=1}u_{ii\gamma})^2.
\end{equation}
For (\ref{eq_21}), we can choose $a_i=\frac{2(1-\epsilon(n))|\lambda_i| (A+\Delta\cdot)}{(1+\lambda_i^2)}$, which can be viewed as the convex case. For (\ref{eq_22}) (\ref{eq_23}), we choose $A=A(\theta)$ large, and $a_i=\min[\frac{2(1-\epsilon(n))\lambda_i (A+\Delta\cdot)}{(1+\lambda_i^2)},n^4](i\le m)$, $a_i=n^4+i(i>m)$ since then $C(\theta)\ge|\lambda_i|\ge|\lambda_n| $, $\frac{\hat\epsilon(\theta) (A+\Delta\cdot)}{(1+\lambda_i^2)}>2n^4$ for large $A(\theta)$. Apply the argument to the quadratic form (\ref{eq_matrix_1}) we have $Q_\gamma\ge 0$. Finally we take the new representation $u_{nn\gamma}=u_{nn\gamma}-\phi_\gamma{ g_{nn}}$ and the reaming term of $\Delta_g$ into account as (\ref{eq_15}) (\ref{eq_150}), we get the weak trace Jacobi inequality
\begin{equation}\label{eq_25}
\Delta_g|A+\Delta\cdot|-(1+\epsilon) \frac{|\nabla_g\Delta\cdot|^2}{A+\Delta\cdot}\ge \hat\epsilon\sum^{n}_{\gamma,i=1}|\lambda_i|g^{ii}g^{ii}u^2_{ii\gamma}- \hat C|A+\Delta\cdot| |D\phi|^2+\Delta \phi,
\end{equation}
where $\epsilon=\epsilon(\frac18)$ and $A=A(\theta),\hat C=\hat C(\theta),\hat\epsilon=\hat\epsilon(\theta)$. Now we go through the remaining steps in the proof of Theorem \ref{Thm_Upbd_n}. Here we refer to Wang-Yuan \cite{WY11}, proof to Theorem 1.1 to treat the supercritical phase case.  We come to the conclusion that
$$
\ln|A+\Delta\cdot|(0)\le C(n,\theta,|Du|_{L^\infty(B_{3}(0))},||\phi||_{C^{0,1} ({B}_{3}(0))}).
$$
Since $u+|x^2|\cot\theta$ is convex, the gradient bound follows and the proof is done.
\end{proof}

\section {interior $C^{2,\alpha}$ regularity on supercritical phase interval}

%In this section we prove an interior $C^{2,\alpha}$  regularity for $C^0$ viscosity solutions of (\ref{gslag1}) on the first phase interval $I_1=(\frac\pi2(n-2),\frac\pi2n)$ by standard approximation.

\subsection{Proof to Theorem \ref{Thm_Upbd_nGloabl}}
\begin{proof}With the aid of the interior estimates, we can prove an interior regularity result for $C^0$ viscosity solutions on the supercritical phase interval $I_1=(\frac\pi2(n-2),\frac\pi2n)$ where holds a comparison principle, which is showed in Cirant-Payne \cite{CP21}, Theorem 6.18, and see also the works of Dinew-Do-T\^o \cite{DDT19}, Harvey-Lawson \cite{HL21}. Since the phase $\phi(x)$ is Lipschitz, we approximate it by a sequence of smooth functions $\phi_{k}(x)$ satisfying the following conditions:
$$
|\phi_{k}-\phi|\le \frac1{k}, \ \phi_{k}\in [\frac\pi2(n-2)+\theta-\frac1{4k},\frac\pi2n-\frac1{4k}],
$$
where $\theta>0$ satisfies $\phi\ge\frac\pi2(n-2)+\theta$ in $B_{\frac78}(0)$, and further
$$
 ||\phi_{k}||_{C^{0,1}(B_{\frac34})}\le C(n)||\phi||_{C^{0,1}(B_{\frac34+o(1)}(0))}.
$$
Next we use the existence and uniqueness results for $C^{2,\alpha}$ solutions in Bhattacharya \cite{Bha20b}, Theorem 1.1, see also Lu \cite{Lu22}, to get approximating solutions of equation (\ref{gslag1}) with corresponding phase $\phi_k$. Specificity, we solve the following Dirichlet problem for large $k$,
$$
\left\{
\begin{array}{ll}
 F(D^2u^k)=\sum^n_{i=1}\arctan \lambda_i(D^2u^{k})=\phi_k(x)   & x\in B_{\frac34}(0),   \\
u^{k}(x)=u(x)    &  x\in \partial B_{\frac34}(0).     \\
\end{array} \right.
$$
Then there exists an unique smooth solution $u^{k}(x) $ in $ B_{\frac34}(0)$. Let $\delta>0$ small, there exists a constant $C=C(n,\theta, \mathrm{osc}_{B_{\frac34}(0)} u,||\phi||_{C^{0,1}(B_{\frac34+o(1)}(0))})$ such that
$$
F(D^2u^{k}+2\delta I_n)= \sum^n_{i=1}\arctan [\lambda_i(D^2u^{k})+2\delta]\ge \phi_k+C\delta,
$$
$$
F(D^2u^{k}-2\delta I_n)= \sum^n_{i=1}\arctan [\lambda_i(D^2u^{k})-2\delta]\le \phi_k-C\delta,
$$
where we use the interior $C^{1,1}$ estimates to $u^k$ by Theorem \ref{Thm_Upbd_critical}, the $L^\infty$ norm of $u^k$ is obtained since $u^k+|x|^2\cot\frac\theta2$ attains its maximum at the boundary when $k$ is large. In view of the discussion precedes Theorem \ref{Thm_Upbd_nGloabl}, we have $C^{2,\alpha}$ estimates for each $0<\alpha<1$,
\begin{equation}\label{eq_31}
||u^k||_{C^{2,\alpha}( B_{\frac12}(0))}\le \hat C(n,\theta,\alpha, \mathrm{osc}_{B_{\frac34}(0)} u,||\phi||_{C^{0,1}(B_{\frac34+o(1)}(0))}).
\end{equation}
Notice that $u^k$ are smooth functions, by comparison principle we have for all $k\ge k_0(C\delta)$ and $m\ge 1$
%In view of Bhattacharya \cite{BMS22}, Theorem 1.1 (and the comment follows the proof), we know that if $|\lambda_{\min}(|Du^k|_{L^\infty})$ has a uniform bound, then $|Du^k|_{L^\infty}$ will be bound by $osc(u^k)$ and $||\phi_k||_{C^{0,1}}$. Since $u\in C^1$ is convex  Notice that $u^k$ are smooth functions, by comparison principle, there exists a integer $k_0=k_0(\delta,||\phi||_{C^{0,1}})$ such that for all $k\ge k_0(\delta)$ and $m\ge 1$
$$
u^{k}+\delta[|x|^2-(\frac34)^2]\le u^{(k+m)}\le u^{k} -\delta[|x|^2-(\frac34)^2].
$$
Then $\{u_k\}$ converges uniformly and up to a subsequence converges in $C^2$ on compact subsets of $B_{\frac34}(0)$. We claim that it converges to $u$, in particular estimate (\ref{eq_31}) holds for $u$. We need the comparison principle for supercritical phase viscosity solutions of (\ref{gslag1}). Similar to the argument for $u^{k+m}$, we have for $k\ge k_0(C\delta)$
$$
u^{k}+\delta[|x|^2-(\frac34)^2]\le u\le u^{k} -\delta[|x|^2-(\frac34)^2].
$$
%Otherwise we can cut $\phi$ to constants for the parts outside of $I_\epsilon$. Then the new function is still Lipschitz, and can be approximated by $u_k$. Finally we let $\epsilon\rightarrow 0$ and get the conclusion.
\end{proof}

\section{Acknowledgements}
The author would like to express his great gratitude for professor Yu Yuan's guidance and advices on this study.
\appendix
\section {Hessian estimates in dimension two}
The two dimensional case was first proved independently by Yu Yuan. Here we prove it to show the stability of our approach. It also serves as the simplest example to make clear the trace Jacobi inequality.
\begin{theorem}\label{Thm_Upbd_2}
Let $u$ be a smooth solution to $\arctan\lambda_1+\arctan\lambda_2=\phi$ in $B_{4}(0)$. Suppose that either u is convex with $\phi(x)\in[0,\frac\pi2)$, or $\phi(x)\in[\theta,\frac\pi2)$ for some constant $\theta>0$. We have
\begin{equation}\label{interior estimate}
|D^2u(0)|\le C,
\end{equation}
where $C$ depends on $\mathrm{osc}_{B_{4}(0)} u, ||\phi||_{C^{0,1}(B_{4}(0))}$, and in the second case also on $\theta$.
\end{theorem}
\begin{proof}
We only verify the weak trace Jacobi inequality, the remaining steps are the same. Cho'ose a coordinate system such that $D^2u$ is diagonal at $p\in \mathcal{M}$, by approximation we suppose that $\lambda_1\neq \lambda_2$. We have at $p$
$$
\lambda_{1,22}=u_{1122}+\frac2{\lambda_1-\lambda_2}u^2_{221}=\lambda_{2,11}+\frac2{\lambda_1-\lambda_2}(u^2_{112}+u^2_{221}),
$$
$$
g^{22}\lambda_{1,22}+g^{11}\lambda_{2,11}=g^{22}\lambda_{2,11}+g^{11}\lambda_{1,22}+2(\lambda_1+\lambda_2)(u^2_{112}+u^2_{221})g^{11}g^{22}.
$$
Differentiate the equation twice, add the above equality by $g^{11}\lambda_{1,11}+g^{22}\lambda_{2,22}$, and use the relation $g^{11}u_{11\gamma}+g^{22}u_{22\gamma}=\phi_\gamma$, we get for some $\delta$ small
\begin{align*}
&\overline \Delta_g\Delta \cdot= \sum^2_{\gamma,i=1}2\lambda_ig^{ii}g^{ii}u^2_{ii\gamma}+2(\lambda_1+\lambda_2)(u^2_{112}+u^2_{221})g^{11}g^{22}+\Delta\phi\\
&\ge\sum^2_{\gamma=1}2(\lambda_1+\lambda_2) g^{11}g^{\gamma\gamma}u_{11\gamma}^2
-\sum^2_{\gamma=1}\delta|\lambda_2|g^{11}g^{11}u_{11\gamma}^2
+\Delta\phi-\frac{C(2)}\delta|A+\Delta\cdot||D\phi|^2,
\end{align*}
\begin{align*}
|\nabla_g \Delta\cdot|^2=\sum^2_{\gamma=1}g^{\gamma\gamma}(u_{11\gamma}+u_{22\gamma})^2
&\le\sum^2_{\gamma=1}(1+\delta)g^{\gamma\gamma}
g^{11}g^{11}(\lambda_1^2-\lambda_2^2)^2 u^2_{11\gamma}\\
&+\frac{C(2)}\delta |A+\Delta\cdot|^2|D\phi|^2.
\end{align*}
We have used a basic fact that $|\lambda_2|\le \lambda_1\le A+\Delta\cdot$ when $A\ge \cot\theta$ in the supercritical phase case. Choose a small $\epsilon=10^{-2}$, we have
\begin{align*}
&\overline \Delta_g|A+\Delta  \cdot|-(1+\epsilon)\frac{|\nabla_g \Delta\cdot|^2}{A+\Delta\cdot}\\
&\ge \sum^2_{\gamma=1}[2(\lambda_1+\lambda_2)-(1+\epsilon+2\delta)\frac{(\lambda_1^2-\lambda_2^2)^2}{A+\Delta\cdot}g^{11}]
g^{11}g^{\gamma\gamma}u^2_{11\gamma}\\
&-\sum^2_{\gamma=1}\delta|\lambda_2|g^{11}g^{11}u_{11\gamma}^2
-\frac{C(2)}\delta|A+\Delta\cdot||D\phi|^2+\Delta\phi.
\end{align*}
For the convex case, we can choose $\delta=\delta(2)$ small and $A=A(2)$ large such that
\begin{equation}\label{eq_A1}
2(\lambda_1+\lambda_2)-(1+\epsilon+2\delta)\frac{(\lambda_1^2-\lambda_2^2)^2}{A+\Delta\cdot}g^{11}\ge \frac12\lambda_1.
\end{equation}
Thus we have
\begin{align*}
\overline \Delta_g|A+\Delta  \cdot|-(1+\epsilon)\frac{|\nabla_g \Delta\cdot|^2}{A+\Delta\cdot}
&\ge \sum^2_{\gamma=1}\frac14\lambda_1g^{11}g^{\gamma\gamma}u_{11\gamma}^2
+\Delta\phi-{C(2)}|A+\Delta\cdot||D\phi|^2\\
&\ge \sum^2_{\gamma,i=1}\frac1{16}|\lambda_i|g^{ii}g^{ii}u_{ii\gamma}^2+
\Delta\phi-{C(2)}|A+\Delta\cdot||D\phi|^2
\end{align*}
For the supercritical phase case with $\lambda_2<0$,
since we assume that $\phi\ge\theta$, we get $\lambda_1\ge \tan\theta$, $\lambda_2>-\cot\theta$, $\lambda_1+\lambda_2= \tan\theta(1-\lambda_1\lambda_2)\ge 2c(\theta)\lambda_1$ for some $c(\theta)>0$. We can choose $A=A(\theta)$ large, $\delta\le 10^{-2}$ small such that
$$
2(\lambda_1+\lambda_2)-(1+\epsilon+2\delta)\frac{(\lambda_1^2-\lambda_2^2)^2}{A+\Delta\cdot}g^{11}\ge \frac12(\lambda_1+\lambda_2)\ge c(\theta)\lambda_1.
$$
Next we choose $\delta=\delta(\theta)$ small such that $\delta \cot\theta<\frac12 c(\theta)\tan\theta$, we have
\begin{align*}
\overline \Delta_g|A+\Delta  \cdot|-(1+\epsilon)\frac{|\nabla_g \Delta\cdot|^2}{A+\Delta\cdot}
\ge \sum^2_{\gamma,i=1}\frac{c(\theta)}{8}|\lambda_i|g^{ii}g^{ii}u_{ii\gamma}^2
+\Delta\phi-{C(\theta)}|A+\Delta\cdot||D\phi|^2.
\end{align*}
Recall $\Delta_g=\overline\Delta_g-\sum^2_{i=1}g^{ii}\lambda_i\phi_i\partial_i$, we get the {weak Jacobi inequality} for a smaller $\epsilon$,
\begin{align*}
 \Delta_g\ln|A+\Delta  \cdot|-\epsilon{|\nabla_g|\ln|A+ \Delta\cdot||^2}\ge \sum^2_{\gamma,i=1}\frac{|\lambda_i|g^{ii}g^{ii}u_{ii\gamma}^2}{K|A+\Delta  \cdot|}
+\frac{\Delta\phi}{A+\Delta  \cdot}-{L}|D\phi|^2,
\end{align*}
where $\epsilon=\frac1210^{-2}$, $A,K,L$ depend on $2$ or $\theta$. The positive quadratic form $u^2_{\clubsuit\clubsuit\spadesuit}$ is to balance the extra terms raising from $\Delta \phi$, just as in the proof of Theorem \ref{Thm_Upbd_n}, Step 2, 3. The remaining steps are the same.

\end{proof}

%\bibliographystyle{amspain}
%\bibliography{References}

\end{document}